\documentclass{amsart}
\newtheorem{Thm}{Theorem}[section]
\newtheorem{Prop}[Thm]{Proposition}

\newtheorem{Lem}[Thm]{Lemma}
\numberwithin{equation}{section}
\begin{document}

\title[$p$-harmonic boundaries and the first reduced $\ell^p$-cohomology]
{The $p$-harmonic boundary for finitely generated groups and the first reduced $\ell^p$-cohomology}

\author[M. J. Puls]{Michael J. Puls}
\address{Department of Mathematics \\
John Jay College-CUNY  \\
445 West 59th Street  \\
New York, NY 10019  \\
USA}
\email{mpuls@jjay.cuny.edu}

\begin{abstract}
Let $p$ be a real number greater than one and let $G$ be a finitely generated, infinite group. In this paper we introduce the $p$-harmonic boundary of $G$. We then characterize the vanishing of the first reduced $\ell^p$-cohomology of $G$ in terms of the cardinality of this boundary. Some properties of $p$-harmonic boundaries that are preserved under rough isometries are also given. We also study the relationship between translation invariant linear functionals on a certain difference space of functions on $G$, the $p$-harmonic boundary of $G$ and the first reduced $\ell^p$-cohomology of $G$.
\end{abstract}


\date{April, 2, 2008}
\maketitle

\section{Introduction}\label{Introduction}
Let $G$ be a finitely generated, infinite group with identity $e_G$ and symmetric generating set $S$. The set of all complex-valued functions on $G$ will be denoted by $\mathcal{F}(G)$, and the set of bounded functions in $\mathcal{F}(G)$ will be denoted by $\ell^{\infty} (G)$. Recall that $\ell^{\infty} (G)$ is a normed space under the sup-norm, so for $f \in \ell^{\infty}(G), \parallel f \parallel_{\infty} = \sup_{g \in G} |f(g) |$. Let $1 \leq p \in \mathbb{R}$ and set 
$$ D_p(G) = \{ f \in \mathcal{F}(G) \mid \sum_{g \in G} |f(gs^{-1}) - f(g) |^p < \infty \mbox{ for all } s \in S \}. $$
Observe that the constant functions, which we identify with $\mathbb{C}$, are in $D_p(G)$. We now define a pseudonorm on $D_p(G)$ by 
$$ \parallel f \parallel = \left( \sum_{s \in S} \sum_{g\in G} | f(gs^{-1}) - f(g) |^p \right)^{1/p} .$$
The set $D_p(G)$ is a reflexive Banach space with respect to the following norm 
$$ \parallel f \parallel_{D_p} = \left( \parallel f \parallel^p + \vert f(e_G) \vert^p \right)^{1/p}.$$
The norm for $D_p(G)$ depends on the symmetric generating set $S$, but the underlying topology does not. If $A \subseteq D_p(G)$, then $(\overline{A})_{D_p}$ will denote the closure of $A$ in $D_p(G)$. Let $\ell^p(G)$ be the set that consists of functions on $G$ for which $\sum_{g \in G} |f(g)|^p$ is finite. Observe that $\ell^p(G)$ is contained in $D_p(G)$. The first reduced $\ell^p$-cohomology space of $G$ is defined by
$$ \bar{H}^1_{(p)} (G) = D_p(G)/ (\overline{ \ell^p(G) \oplus \mathbb{C}})_{D_p}.$$
There has been some work done relating various boundaries of $G$ and the nonvanishing of $\bar{H}^1_{(p)} (G)$. It was shown in Chapter 8, section C2 of \cite{Gromov} (also see \cite{Elek} ) that if the $\ell_p$-corona of $G$ contains more than one element, then $\bar{H}^1_{(p)} (G) \neq 0$. In \cite{PulsADM} it was shown that if there is a Floyd boundary of $G$ containing more than two elements, and if the Floyd admissible function satisfies a certain decay condition, then $\bar{H}^1_{(p)} (G) \neq 0$. However, it is unknown if the converse of either of these two results is true. In this paper we introduce the $p$-harmonic boundary of a finitely generated group, see Section \ref{Mainresults} for the definition, which characterizes the vanishing of $\bar{H}^1_{(p)} (G)$ in terms of the cardinality of this boundary. We also prove results concerning the $p$-harmonic boundary and rough isometries. We conclude the paper by giving a link between the $p$-harmonic boundary of $G$ and continuous translation invariant linear functionals on a certain difference space of functions on $G$. This paper was inspired by \cite{Lee}. I would like to thank Peter Linnell for many useful comments on a preliminary version of this paper. 

\section{Outline of Paper and Statement of Main Results}\label{Mainresults}
We begin this section by defining the main object of study in this paper, the $p$-harmonic boundary. Let $ 1 \leq p \in \mathbb{R}$ and let $BD_p(G)$ be the set of bounded functions in $D_p(G)$. The set $BD_p(G)$ is a Banach space under the norm $\parallel f \parallel_{BD_p} = \parallel f \parallel_{\infty} + \parallel f \parallel$, where $f \in BD_p(G)$. For $X \subseteq BD_p(G), (\overline{X})_{BD_p}$ will denote the closure of $X$ in $BD_p(G)$. Under the usual operations of pointwise multiplication, addition and scalar multiplication, $BD_p(G)$ is an abelian Banach algebra. A character on $BD_p(G)$ is a nonzero homomorphism from $BD_p(G)$ into the complex numbers. We denote by $Sp(BD_p(G))$ the set of characters on $BD_p(G)$. With respect to the weak $\ast$-topology, $Sp(BD_p(G))$ is a compact Hausdorff space. The space $Sp(BD_p(G))$ is known as the spectrum of $BD_p(G)$. Let $C(Sp(BD_p(G)))$ denote the set of continuous functions on $Sp(BD_p(G))$. For each $f \in BD_p(G)$ a continuous function $\hat{f}$ can be defined on $Sp(BD_p(G))$ by $\hat{f}(\tau) = \tau(f)$. The map $f \rightarrow \hat{f}$ is known as the Gelfand transform.

Define a map $i \colon G \rightarrow Sp(BD_p(G))$ by $(i(g))(f) = f(g)$. For $g \in G$, define $\delta_g$ by $\delta_g(x) = 0$ if $x \neq g$ and $\delta_g (g) = 1$. Let $g, h \in G$ and suppose $i(g) = i(h)$, then $(i(g))(\delta_g) = (i(h))(\delta_g)$ which implies $\delta_g(g) = \delta_g(h)$. Thus $i$ is an injection. If $f$ is a nonzero function in $BD_p(G)$, then there exists a $g \in G$ such that $\hat{f}(i(g)) \neq 0$ since $\hat{f}(i(g)) = f(g)$. Hence $BD_p(G)$ is semisimple. Theorem 4.6 on page 408 of \cite{TaylorLay} now tells us that $BD_p(G)$ is isomorphic to a subalgebra of $C(Sp(BD_p(G)))$ via the Gelfand transform. Since the Gelfand transform separates points of $Sp(BD_p(G))$ and the constant functions are contained in $BD_p(G)$, the Stone-Weierstrass Theorem yields that $BD_p(G)$ is dense in $C(Sp(BD_p(G)))$ with respect to the sup-norm. The following proposition, which is essentially \cite[Proposition 1.1]{Elek}, shows that $i(G)$ is dense in $Sp(BD_p(G))$.

\begin{Prop} \label{dense}
The image of $G$ under $i$ is dense in $Sp(BD_p(G))$.
\end{Prop}
\begin{proof}
Let $K$ be the closure of the image of $i$. Suppose that $K \neq Sp(BD_p(G))$. By Urysohn's Lemma there exists a nonzero element $h \in C(Sp(BD_p(G)))$ such that $h\mid_K = 0$. Let $(f_n)$ be a sequence in $ BD_p(G)$ that converges to $h$ in the sup-norm. So given $\epsilon > 0$ there exists a number $N$ such that $\vert \widehat{f_n} (i(g)) - h(i(g)) \vert < \epsilon$ for all $g \in G$ and for all $n > N$. Consequently, $(f_n) \rightarrow 0$ in $BD_p(G)$ with respect to the sup-norm. Thus $h = 0$, which is a contradiction.
\end{proof}
When the context is clear we will abuse notation and write $G$ for $i(G)$ and $g$ for $i(g)$. The compact Hausdorff space $Sp(BD_p(G)) \backslash G$ is known as the $p$-Royden boundary of $G$. When $p =2$ this is simply known as the Royden boundary of $G$. Let $\mathbb{C}G$ be the set of functions on $G$ with finite support and let $B(\overline{\mathbb{C}G})_{D_p}$ denote the set of bounded functions in $(\overline{\mathbb{C}G})_{D_p}$. Suppose $(f_n)$ is a sequence in $B(\overline{\mathbb{C}G})_{D_p}$ that converges to a bounded function $f$ in the $BD_p(G)$-norm. It now follows from $\parallel f - f_n \parallel_{D_p} \leq \parallel f - f_n \parallel_{BD_p}$ that $f \in (\overline{\mathbb{C}G})_{D_p}$. Thus $B(\overline{\mathbb{C}G})_{D_p}$ is closed in $BD_p(G)$ with respect to the $BD_p(G)$-norm. The {\em $p$-harmonic boundary} of $G$ is the following subset of the $p$-Royden boundary
$$\partial_p(G) \colon= \{ x \in Sp(BD_p(G)) \backslash G \mid \hat{f}(x) = 0 \mbox{ for all } f \in B(\overline{\mathbb{C}G})_{D_p}\}.$$
In this paper we will use the notation $\#(A)$ to mean the cardinality of a set $A$.

In Section \ref{pharmbound} we prove various results concerning $\partial_p(G)$. With these results in hand we will be able to prove the following characterization concerning the vanishing of $\bar{H}^1_{(p)} (G)$. 

\begin{Thm} \label{charvanishing}
Let $1 < p \in \mathbb{R}$. Then $\bar{H}^1_{(p)}(G) \neq 0$ if and only if $\#(\partial_p(G)) > 1$.
\end{Thm}

We will conclude Section \ref{pharmbound} by describing a neighborhood base for the topology on $\partial_p(G)$.

Let $X$ denote the Cayley graph of $G$ with respect to the generating set $S$. Thus the vertices of $X$ are elements of $G$, and $g_1, g_2 \in G$ are joined by an edge if and only if $g_1 = g_2s^{\pm 1}$ for some generator $s$. We make $X$ into a metric space by assigning length one to each edge, and defining the distance $d_S(g_1, g_2)$ between any two vertices $g_1, g_2$ in $X$ to be the length of the shortest path between $g_1$ and $g_2$. The metric $d_S$ on $X$ is known as the word metric. For the rest of this paper we will drop the use of the subscript $S$ and $d(g_1, g_2)$ will always denote the distance between $g_1$ and $g_2$ in the word metric. We will denote $d(e_G, g)$ by $|g|$ for $g \in G$. From now on we will implicitly assume that all finitely generated groups are equipped with the word metric. 

Let $(X, d_X)$ and $(Y, d_Y)$ be metric spaces. A map $\phi \colon X \rightarrow Y$ is said to be a {\em rough isometry} if it satisfies the following two conditions:
\begin{enumerate}
\item There exists constants $ a \geq 1, b \geq 0$ such that for $x_1, x_2 \in X$
$$ \frac{1}{a} d_X (x_1, x_2) - b \leq d_Y (\phi(x_1), \phi(x_2)) \leq a d_X(x_1, x_2) + b  $$
\item There exists a positive constant $c$ such that for each $y \in Y$, there exists an $x \in X$ that satisfies $d_Y ( \phi(x), y) < c$.
\end{enumerate}

For a rough isometry $\phi$ there exists a rough isometry $\psi \colon Y \rightarrow X$ such that if $x \in X, y \in Y$, then $d_X( (\psi \circ \phi)(x), x) \leq a(c+b)$ and $d_Y((\phi \circ \psi)(y), y) \leq c$. The map $\psi$, which is not unique, is said to be a rough inverse for $\phi$. Whenever we refer to a rough inverse to a rough isometry in this paper, it will always satisfy the above conditions. In Section \ref{homeomorphicproof} we will prove
\begin{Thm} \label{homeomorphic} Let $G$ and $H$ be finitely generated groups. If there is a rough isometry from $G$ to $H$, then $\partial_p(G)$ is homeomorphic to $\partial_p(H)$.
\end{Thm}

Let $f \in {\mathcal F}(G)$ and let $g \in G$. Define 
$$ \triangle_p f(g) \colon = \sum_{s \in S} |f(gs^{-1}) - f(g) |^{p-2} (f(gs^{-1}) - f(g)).$$
In the case $1 < p < 2$, we make the convention that $|f(gs^{-1}) - f(g)|^{p-2}(f(gs^{-1}) - f(g)) =0$ if $f(gs^{-1}) = f(g)$. We shall say that $f$ is {\em $p$-harmonic} if $f \in D_p (G)$ and $\triangle_p f(g) = 0$ for all $g \in G$. Denote the set of $p$-harmonic functions on $G$ by $HD_p(G)$. The set of bounded functions in $HD_p(G)$ will be denoted by $BHD_p(G)$. Observe that the constant functions are in $BHD_p(G).$ We will conclude Section \ref{homeomorphicproof} by proving
\begin{Thm}\label{bijection}
Let $G$ and $H$ be finitely generated, infinite groups. If $\phi$ is a rough isometry from $G$ to $H$, then there is a bijection from $BHD_p(G)$ to $BHD_p(H)$.
\end{Thm}

Let $E$ be a normed space of functions on $G$. Let $f \in E$ and let $x \in G$. The right translation of $f$ by $x$, denoted by $f_x$, is the function $f_x(g) = f(gx^{-1})$. Assume that if $f \in E$ then $f_x \in E$ for all $x\in G$; that is, that $E$ is right translation invariant. For the rest of this paper translation invariant will mean right translation invariant. We shall say that $T$ is a translation invariant linear functional (TILF) on $E$ if $T(f_x) = T(f)$ for $f \in E$ and $x \in G$. We will use TILFs to denote translation invariant linear functionals. A common question to ask is that if $T$ is a TILF on $E$, then is $T$ continuous? For background about the problem of automatic continuity see \cite{ Meisters, Saeki, Willis2, Woodward}. Define
$$ \mbox{Diff}(E) := \mbox{ linear span}\{f_x - f \mid f \in E, x \in G \}. $$
It is clear that $\mbox{Diff}(E)$ is contained in the kernel of any TILF on $E$. In Section \ref{tilf} we study TILFs on $D_p(G)/\mathbb{C}$. In particular we prove

\begin{Thm} \label{chartilf} 
Let $G$ be a finitely generated infinite group and let $1 < p \in \mathbb{R}$. Then $\#(\partial_p(G)) > 1$ if and only if there exists a nonzero continuous TILF on $D_p(G)/\mathbb{C}.$
\end{Thm}

It was shown in \cite{Willis1} that if $G$ is nonamenable, then the only TILF on $\ell^p(G)$ is the zero functional. (Consequently every TILF is automatically continuous!). We will conclude Section \ref{tilf} by showing that this result is not true for $D_p(G)/\mathbb{C}$.

\section{Proof of Theorem \ref{charvanishing} and other results concerning the $p$-harmonic boundary} \label{pharmbound}
In this section we will prove Theorem \ref{charvanishing}. We will conclude this section by describing a base for the topology on $\partial_p(G)$. Before we prove Theorem \ref{charvanishing} we prove some preliminary results concerning $\partial_p(G)$.  We begin with the following lemma that will be needed in the sequel

\begin{Lem} \label{infinitelength}
If $x \in \partial_p(G)$ and $(g_n)$ is a sequence in $G$ that converges to $x$, then $|g_n| \rightarrow \infty$ as $n \rightarrow \infty$.
\end{Lem}
\begin{proof}
Let $x \in \partial_p(G)$ and suppose $(g_n) \rightarrow x$, where $(g_n)$ is a sequence in $G$. Let $B \in \mathbb{R}$ and define a function $\chi_B$ on $G$ by $\chi_B (g) = 1$ if $|g| \leq B$ and $\chi_B (g) = 0$ if $|g| > B$. Since $\chi_B$ has finite support it is an element of $\mathbb{C}G$. Suppose there exists a real number $M$ such that $|g_n| \leq M$ for all $n$. Then $\widehat{\chi_M} (x) = \lim_{n \rightarrow \infty} \chi_M (g_n) = 1$, a contradiction. Thus $|g_n| \rightarrow \infty$ as $n \rightarrow \infty$.
\end{proof}
Denote the constant function that always takes the value 1 on $G$ by $1_G$.

\begin{Prop}\label{noelement}
Let $1 < p \in \mathbb{R}$. Then $\partial_p(G) = \emptyset$ if and only if $1_G \in B(\overline{\mathbb{C}G})_{D_p}$.
\end{Prop}
\begin{proof}
Suppose $\partial_p(G) \neq \emptyset$ and let $x \in \partial_p(G)$. Now $\widehat{1_G}(x) = 0$ because $1_G \in B(\overline{\mathbb{C}G})_{D_p}$. Let $(g_n)$ be a sequence in $G$ that converges to $x$ in $Sp(BD_p(G))$. Then $\widehat{1_G} (x) = \lim_{n \rightarrow \infty} \widehat{1_G} (g_n) = 1$, a contradiction. Hence $\partial_p(G) = \emptyset $.

Conversely, suppose that $1_G \notin B(\overline{\mathbb{C}G})_{D_p}$. It is easy to verify that $B(\overline{\mathbb{C}G})_{D_p}$ is an ideal in $BD_p(G)$. Let $M$ be a maximal ideal in $BD_p(G)$ that contains $B(\overline{\mathbb{C} G})_{D_p}$. Thus there exists $x \in Sp(BD_p(G))$ with $\mbox{ker}(x) = M$. So $\hat{f}(x) = x(f) = 0$ for all $f\in B(\overline{\mathbb{C}G})_{D_p}$. For each $g \in G$ there exists $f \in \mathbb{C}G$ (in particular $ \delta_g$) such that $g (f) = f(g) \neq 0$ which means that $x$ cannot be in $G$. Hence $\partial_p(G) \neq \emptyset.$
\end{proof}

For the rest of this paper, unless otherwise stated, we will assume that $1_G \notin B(\overline{\mathbb{C}G})_{D_p}$. 

Let $f$ and $h$ be elements in $BD_p(G)$ and let $1 < p \in \mathbb{R}$. Define
$$ \langle \triangle_p h, f \rangle \colon = \sum_{g \in G} \sum_{s\in S} | h(gs^{-1}) - h(g)|^{p-2} (h(gs^{-1}) - h(g))(\overline{ f(gs^{-1}) - f(g)}),$$
where $\overline{f(gs^{-1}) - f(g)}$ denotes the complex conjugate. The above sum exists since $\sum_{g \in G} \sum_{s \in S} ||h(gs^{-1}) - h(g)|^{p-2} (h(gs^{-1}) - h(g))|^q < \infty$ where $\frac{1}{p} + \frac{1}{q} = 1$. The following propositions will be needed in the sequel. For their proofs see Section 3 of \cite{Puls}, where similar propositions for $D_p(G)$ were proven.

\begin{Prop}\label{diffconst} Let $f_1$ and $f_2$ be functions in $BD_p(G).$ Then $\langle \triangle_p f_1 - \triangle_p f_2, f_1 - f_2 \rangle = 0$ if and only if $f_1(gs^{-1}) - f_1(g) = f_2(gs^{-1}) - f_2(g)$ for all $g \in G$ and for all $s \in S$.
\end{Prop}

Note that the condition $f_1(gs^{-1}) - f_1(g) = f_2(gs^{-1}) - f_2(g)$ for all $g \in G$ and for all $s \in S$ implies that $(f_1 - f_2)(gs^{-1}) = (f_1 - f_2)(g)$ for all $g \in G$ and for all $s \in S$. Hence $f_1 - f_2$ is constant on $G$ since $S$ generates $G$. Now let $B(\overline{\ell^p(G)})_{D_p}$ be the set of bounded functions in $(\overline{\ell^p(G)})_{D_p}$. The set $B(\overline{\ell^p(G)})_{D_p}$ is closed in $BD_p(G)$ with respect to the $BD_p(G)$-norm for the same reason that $B(\overline{\mathbb{C}G})_{D_p}$ is closed in $BD_p(G)$.

\begin{Prop}\label{decomp}
Let $1 < p \in \mathbb{R}$ and suppose $B(\overline{\ell^p(G)})_{D_p} \neq BD_p(G)$. If $f \in BD_p(G)$, then we can write $f = u + h$, where $u \in B(\overline{\ell^p(G)})_{D_p}$ and $h \in BHD_p(G).$ This decomposition is unique up to a constant function.
\end{Prop}

Recall that the $\ell^p$-norm on $\ell^p(G)$ is given by $\parallel f \parallel_p^p = \sum_{g \in G} |f(g)|^p$, where $f \in \ell^p(G)$. We are now ready to prove

\begin{Thm} \label{vanishbound} Let $f \in BD_p(G)$. Then $f \in B(\overline{\ell^p(G)})_{D_p}$ if and only if $\hat{f}(x) = 0$ for all $x \in \partial_p(G).$
\end{Thm}
\begin{proof}
Since $\mathbb{C}G$ is dense in $\ell^p(G)$ in the $\ell^p$-norm, it follows that $\ell^p(G) \subseteq (\overline{\mathbb{C}G})_{D_p}$. Thus $B(\overline{\ell^p(G)})_{D_p} = B(\overline{\mathbb{C}G})_{D_P}$. There exists a sequence $(f_n)$ in $\mathbb{C}G$ such that $\parallel f - f_n \parallel_{BD_p} \rightarrow 0$ because $B(\overline{\ell^p(G)})_{D_p}$ is closed in $BD_p(G)$ with respect to the $BD_p(G)$-norm. Due to the density of $G$ in $Sp(BD_p(G))$ we see that $\widehat{f_n} \rightarrow \hat{f}$  in $C(Sp(BD_p(G)))$ with respect to the sup-norm. So if $x \in \partial_p(G)$, then $\hat{f}(x) = 0$ since $\widehat{f_n}(x) = 0$ for all $n$.

Conversely, suppose $\hat{f}(x) = 0$ for all $x \in \partial_p (G)$. By Proposition \ref{decomp} we can write $f = u + h$, where $u \in B(\overline{\ell^p(G)})_{D_p}$ and $h \in BHD_p(G)$. So $\hat{h}(x) = 0$ for all $x \in \partial_p(G)$ since $\hat{u}(x) = 0.$ Now $|h(g)|$ attains its maximum value on $G$ since $|\hat{h}(g)|$ attains its maximum value on $Sp(BD_p(G))$ and $G$ is dense in $Sp(BD_p(G))$. Let $g \in G$ with $|h(g)| \geq |h(a)|$ for all $a \in G$. It follows $h(g) = h(gs^{-1})$ for all $s \in S$ since $\sum_{s \in S} |h(gs^{-1}) - h(g)|^{p-2} h(gs^{-1}) = \sum_{s \in S} |h (gs^{-1}) - h(g) |^{p-2} h(g)$. We now obtain $h(g) = h(a)$ for all $a \in G$ since $S$ generates $G$. Therefore $h =0$ and consequently $f = u \in B(\overline{\ell^p(G)})_{D_p}.$
\end{proof}

We now show that a function in $BHD_p(G)$ is uniquely determined by its values on $\partial_p(G)$.

\begin{Prop} \label{boundval}
 A function in $BHD_p(G)$ is uniquely determined by its values on $\partial_p(G)$.
\end{Prop}

\begin{proof}
Let $h_1$ and $h_2$ be elements of $BHD_p(G)$ with $\widehat{h_1}(x) = \widehat{h_2}(x)$ for all $x \in \partial_p(G).$ Since $\widehat{h_1} - \widehat{h_2} = 0$ on $\partial_p(G)$ and $h_1 - h_2 \in BD_p(G), h_1 - h_2 \in B(\overline{\ell^p(G)})_{D_p}$ by Theorem \ref{vanishbound}. So there exists a sequence $(f_n)$ in $\ell^p(G)$ that converges to $h_1 - h_2$. Using Proposition 3.4 of \cite{Puls} we obtain $\langle \triangle_p h_1 - \triangle_p h_2, h_1 - h_2 \rangle = \lim_{n \rightarrow \infty} \langle \triangle_ph_1 - \triangle_p h_2, f_n \rangle = 0$. It now follows from Proposition \ref{diffconst} that $h_1-h_2$ is constant on $G$. Hence, $h_1 - h_2 = 0$ since $\widehat{h_1}(x) - \widehat{h_2}(x) = 0$ for all $x \in \partial_p(G)$.
\end{proof}

We now give a characterization of when $BHD_p(G)$ is precisely the constant functions in terms of the cardinality of $\partial_p(G)$.

\begin{Thm} \label{plious}
 Let $1 < p \in \mathbb{R}$. Then $BHD_p(G) \neq \mathbb{C}$ if and only if $\#(\partial_p(G)) > 1$.
\end{Thm}

\begin{proof}
Suppose that $\# (\partial_p(G))  =1$ and that $x \in \partial_p(G)$. Let $h \in BHD_p(G)$. Then $\hat{h}(x) = c$, for some constant $c$. It now follows from Proposition \ref{boundval} that the function $h(g) = c$ for all $g \in G$ is the only function in $BHD_p(G)$ with $\hat{h}(x) = c$. Hence $BHD_p(G) = \mathbb{C}$.

Conversely, suppose $\#(\partial_p(G) ) > 1$. Let $x, y \in \partial_p(G)$ such that $x \neq y$ and pick $f \in BD_p(G)$ that satisfies $x(f) \neq y(f)$. Thus $\hat{f}(x) \neq \hat{f}(y)$. It now follows from Theorem \ref{vanishbound} that $f \notin B(\overline{\ell^p(G)})_{D_p}$. By combining Theorem \ref{decomp} with Theorem \ref{vanishbound} we obtain an $h \in BHD_p(G)$ with $\hat{h}(z) = \hat{f}(z)$ for all $z \in \partial_p(G)$. Since $G$ is dense in $Sp(BD_p(G))$ there exists sequences $(x_n)$ and $(y_n)$ in $G$ such that $(x_n)(h) \rightarrow x(h)$ and $(y_n)(h) \rightarrow y(h)$. Hence $\lim_{n \rightarrow \infty} h(x_n) = x(h) \neq y(h) = \lim_{n \rightarrow \infty} h(y_n)$. Therefore, $h$ is not constant on $G$, that is $BHD_p(G) \neq \mathbb{C}$. 
\end{proof}

We now prove Theorem \ref{charvanishing}. Suppose $\partial_p(G) = \emptyset$. By Proposition \ref{noelement} there exists a sequence $(f_n)$ in $\mathbb{C}G$ with $\parallel f_n -1_G \parallel_{D_p} \rightarrow 0$. It follows that $\parallel f_n \parallel \rightarrow 0$ and $(f_n (e_G)) \not\rightarrow 0$. Thus $\bar{H}^1_{(p)} (G) = 0$ by Theorem 3.2 of \cite{PulsCMB}. We now assume $\partial_p(G) \neq \emptyset$. It was shown in \cite[Theorem 3.5]{Puls} that $\bar{H}^1_{(p)}(G) \neq 0$ if and only if $HD_p(G) \neq \mathbb{C}$. Let $X$ be the Cayley graph of $G$ with respect to the generating set $S$. Each vertex in $X$ has degree $\#(S)$, so by Lemma 4.4 of \cite{H2} we have that $BHD_p(G) = \mathbb{C}$ if and only if $HD_p(G) = \mathbb{C}$. Theorem \ref{charvanishing} now follows from Theorem \ref{plious}. 

We will conclude this section by giving a neighborhood base for the topology on $\partial_p(G)$. Let $A$ be a subset of $G$ and define
$$ \partial A := \{ g \in G\backslash A \mid \mbox{ there exists } s \in S \mbox{ with } gs^{-1} \in A \}. $$
Let $A$ be an infinite connected subset of $G$ with $\partial A \neq \emptyset$. The set $A$ is called a {\em $D_p$-massive subset} if there exists a nonnegative function $u \in BD_p(G)$ that satisfies the following:
\begin{enumerate}
\item  $\triangle_p u (a) = 0 \mbox{ for all } a \in A$
\item  $u(g) = 0 \mbox{ for } g \in \partial A $
\item  $\sup_{ a \in A} u(a) = 1$
\end{enumerate} 

A function $u$ that satisfies the above conditions is called an {\em inner potential} of the $D_p$-massive subset $A$. For a $D_p$-massive subset $A$ of $G$ define $\hat{A}_u = \{ \tau \in Sp(BD_p(G)) \mid \hat{u}(\tau) > 0 \}$, where $u$ is an inner potential of $A$ that satisfies $u(g) = 0$ for all $g \in G\setminus A$. Set
$$ {\mathcal B} = \{ \hat{A}_u \cap \partial_p(G) \mid A \mbox{ is a }D_p\mbox{-massive subset of }G \}.$$
We now show that ${\mathcal B}$ is a base for the topology on $\partial_p(G)$.

\begin{Prop} \label{base}
The set ${\mathcal B}$ is a base for the topology on $\partial_p(G)$.
\end{Prop}

\begin{proof}
Let $U$ be an open subset of $\partial_p(G)$ and let $\tau \in U$. By Urysohn's lemma there exists an $f \in C(Sp(BD_p(G)))$ with $f(\tau) =1$ and $f = 0$ on $\partial_p(G) \setminus U$. Since the Gelfand transform of $BD_p(G)$ is dense in $C(Sp (BD_p(G)))$ with respect to the sup-norm we will assume that $f \in BD_p(G)$. By Proposition \ref{decomp} we can write $f = w + h$, where $w \in B(\overline{\ell^p(G)})_{D_p}$ and $h \in BHD_p(G)$. Since $\hat{w} = 0$ on $\partial_p(G)$, we have that $\hat{h}(\tau) = 1$ and $\hat{h} = 0$ on $\partial_p(G) \setminus U$. Also $0 \leq \hat{h} (\sigma) \leq 1$ for all $\sigma \in Sp(BD_p(G))$ by the maximum principle. Fix $\epsilon$ with $0 < \epsilon < 1$ and set $W = \{ g \in G \mid h(g) > \epsilon \}$. The set $W$ is a infinite proper subset of $G$ since $\hat{h}$ is nonconstant and continuous. Let $A$ be a component of $W$ that contains a sequence $(g_k)$ that converges to $\tau$. It follows from Theorem 3.14 of \cite{H2} that $A$ is infinite. Define a function $v$ on $G$ by
$$ v = \frac{h - \epsilon}{1- \epsilon}. $$
Let $O_n = \{ g \in G \mid |g| < n \}$. By Theorem 3.11 of \cite{H2} there exists a $p$-harmonic function $u_n$ on $O_n \cap A$ that takes the values max$\{ 0, v \}$ on $G \setminus (O_n \cap A)$ such that $0 \leq u_n \leq 1$ on $O_n \cap A$. Now $u_n \in BD_p(G)$ since $(O_n \cap A )$ is finite. By passing to a subsequence if necessary we may assume that the sequence $(u_n)$ converges pointwise to a function $u$. By Lemma 3.21 of \cite{H2} $u$ is $p$-harmonic on $A$. Also $v \leq u_n \leq 1$ on $O_n$ so $\sup_{a \in A} u(a) = 1$. Furthermore $u = 0$ on $\partial A$ because $h \leq \epsilon$ on $\partial A$. Since $h \in BD_p(G)$ it follows that $u \in BD_p(G)$. Thus $A$ is a $D_p$-massive subset of $G$ with inner potential $u$.

Define a function $u'$ on $G$ by $u' =u$ on $A$ and $u'=0$ on $G \setminus A$. Clearly $u'$ is an inner potential for $A$. We will now show that $\hat{A}_{u'} \cap \partial_p(G)$ is a basic neighborhood of $\tau$ contained in $U$. Recall that $A$ contains a sequence $(g_k)$ that converges to $\tau$. Since $\hat{h} (\tau) = 1$ it follows that $\lim_{k \rightarrow \infty} v(g_k) = 1$. Thus $\widehat{u'} (\tau) = 1$ due to $v \leq u \leq 1$. So $\tau \in \hat{A}_{u'} \cap \partial_p(G)$. The proof will be complete once we show that $\hat{A}_{u'} \cap \partial_p(G) \subseteq U$, which we now do. Let $x \in \partial_p(G) \cap \hat{A}_{u'}$ and let $(g_k)$ be a sequence in $G$ that converges to $x$. For large $k$ we have $\widehat{u'} (g_k) > 0$ since $\widehat{u'} (x) >0$. It follows that $g_k \in A$ for large $k$. Therefore, $x \in U$ because $f(x) = \hat{h}(x) = \lim_{k \rightarrow \infty} h(g_k) \geq \epsilon.$
\end{proof}

\section{Proofs of Theorems \ref{homeomorphic} and \ref{bijection}}\label{homeomorphicproof}
In this section $G$ and $H$ will always be finitely generated groups with symmetric generating sets $S$ and $T$ respectively. We will denote the Cayley graph of $G$ by $X_G$ and the Cayley graph of $H$ by $X_H$. Let $\phi \colon G \rightarrow H$ be a rough isometry and let $\phi^{\ast}$ denote the map from $\ell^{\infty}(H)$ to $\ell^{\infty}(G)$ given by $\phi^{\ast} f(g) = f (\phi(g))$, where $f \in \ell^{\infty}(H)$. We start by defining a map $\bar{\phi} \colon \partial_p(G) \rightarrow \partial_p(H)$. Let $x \in \partial_p(G)$, then there exists a sequence $(g_n)$ in $G$ such that $(g_n) \rightarrow x$. Now $(\phi (g_n))$ is a sequence in the compact Hausdorff space $Sp (BD_p(H))$. By passing to a subsequence if necessary we may assume that $(\phi (g_n))$ converges to a unique limit $y$ in $Sp (BD_p(H))$. Now define $\bar{\phi} (x) = y$. Before we show that $y \in \partial_p (H)$ and $\bar{\phi}$ is well-defined we need the following lemma.
\begin{Lem} \label{adjoint}
Let $G$ and $H$ be finitely generated groups. If $\phi \colon G \rightarrow H$ is a rough isometry, then 
\begin{enumerate} 
  \item[]
  \begin{enumerate}
      \item $\phi^{\ast} \colon BD_p(H) \mapsto BD_p(G)$
      \item $\phi^{\ast} \colon \ell^p(H) \mapsto \ell^p(G) $
      \item $\phi^{\ast} \colon B(\overline{\ell^p(H)})_{D_p} \mapsto B(\overline{\ell^p(G)})_{D_p}.$
   \end{enumerate} 
 \end{enumerate}
\end{Lem}

\begin{proof}
We will only prove part (a) since the proofs of parts (b) and (c) are similar. Let $f \in BD_p(H)$. We will now show that $\phi^{\ast} f \in BD_p(G)$. Let $g \in G$ and $s \in S$, so $g$ and $gs^{-1}$ are neighbors in $X_G$ but $\phi(gs^{-1})$ and $\phi(g)$ are not necessarily neighbors in $X_H$. However by the definition of rough isometry there exists constants $a \geq 1$ and $b \geq 0$ such that $d_H( \phi(gs^{-1}), \phi(g)) \leq a+b$. Set $h_1 = \phi (g)$ and $h_l = \phi(gs^{-1})$ and let $ h_1, \dots, h_l $ be a path in $X_H$ with length at most $a + b$. Thus
\begin{equation}
\begin{split}
 |\phi^{\ast}f(gs^{-1}) - \phi^{\ast}f(g) |^p & =  | f(\phi(gs^{-1})) - f(\phi(g)) |^p \\
  & \leq |a +b|^{p-1} \sum_{j=1}^{l-1} | f(h_{j+1}) - f(h_j)|^p. \label{eq:boundaplusb}
\end{split}
\end{equation}
The above inequality follows from Jensen's inequality applied to the function $x^p$ for $x >0$.

Let $h \in H$ and $t \in T$. We now claim that there is at most a finite number of paths in $X_H$ of length at most $a + b$ that contain the edge $h, ht^{-1}$ and have the endpoints $\phi(g)$ and $\phi(gs^{-1})$. To see the claim let $U$ be the set of all elements in $G$ such that the following four distances: $d_H (\phi(g), h), d_H (\phi(g), ht^{-1}), d_H(\phi(gs^{-1}), h)$ and $d_H(\phi(gs^{-1}), ht^{-1})$ are all at most $a + b$. Let $g$ and $g'$ be elements in $U$. By the triangle inequality, $d_H ( \phi(g'), \phi (g)) \leq d_H(\phi (g'), h) + d_H ( \phi(g), h)$. It now follows from the definition of rough isometry that $d_G (g', g) \leq 2a^2 + 3ab$. Thus the metric ball $B(g, 2a^2 + 3ab +1)$ contains $U$ as a subset. Hence the cardinality of $U$ is bounded above by some constant $k$. Observe that $k$ is independent of $h$ and $t$. Since $f \in BD_p(H)$ it follows from \ref{eq:boundaplusb} that
$$ \sum_{g \in G} \sum_{s \in S} | \phi^{\ast} f(gs^{-1}) - \phi^{\ast} f(g) |^p \\
\leq |a+b|^{p-1} k \sum_{h \in H} \sum_{t \in T} | f(ht^{-1}) - f(h) |^p \\
< \infty. $$
\end{proof}

We are now ready to prove 
\begin{Prop} \label{welldefined}
The map $\bar{\phi}$ is well-defined from $\partial_p(G)$ to $\partial_p(H)$.
\end{Prop}
\begin{proof}
Let $x, y$ and $(g_n)$ be as above. We start by showing that $y \in \partial_p(H)$. Lemma \ref{infinitelength} tells us that $|g_n| = d_G(e_G, g_n) \rightarrow \infty$ as $n \rightarrow \infty$. The element $\phi (e_G)$ is fixed in $H$ so it follows from the definition of rough isometry that $d_H ( \phi(e_G), \phi (g_n)) \rightarrow \infty$ as $n \rightarrow \infty$. Thus $ y \in Sp(BD_p(H)) \setminus H$ since $y = \lim_{n \rightarrow \infty} \phi (g_n) \notin H$. Let $f \in B(\overline{\ell^p(H)})_{D_p}$ and suppose $\hat{f}(y) \neq 0$. Then $0 \neq \lim_{n \rightarrow \infty} f( \phi (g_n)) = \phi^{\ast} f(x)$. By part (c) of Lemma \ref{adjoint} $\phi^{\ast} f \in B(\overline{\ell^p(G)})_{D_p}$ and Theorem \ref{vanishbound} says that $\phi^{\ast} f(x) = 0$, a contradiction. Hence $\hat{f}(y) =0 $ for all $f \in B(\overline{\ell^p(H)})_{D_p}$, so $y \in \partial_p (H)$.

We will now show that $\bar{\phi}$ is well-defined. Let $(g_n)$ and $(g_n')$ be sequences in $G$ that both converge to $x \in \partial_p(G)$. Now suppose that $(\phi (g_n))$ converges to $y_1$ and $(\phi (g_n'))$ converges to $y_2$ in $Sp( BD_p(H))$. Assume that $y_1 \neq y_2$ and let $f \in BD_p(H)$ such that $f(y_1) \neq f(y_2)$. By  part(a) of Lemma \ref{adjoint} $\phi^{\ast} f \in BD_p(G)$. Thus 
$$ \lim_{n \rightarrow \infty} \phi^{\ast} f(g_n) = \phi^{\ast} f(x) = \lim_{n \rightarrow \infty} \phi^{\ast} f(g_n')$$
which implies $f(y_1) = f(y_2)$, a contradiction. Hence $\bar{\phi}$ is a well-defined map from $\partial_p(G)$ to $\partial_p(H)$.
\end{proof}

The next lemma will be used to show that $\bar{\phi}$ is one-to-one and onto.
\begin{Lem} \label{ontolemma}
Let $\phi \colon G \rightarrow H$ be a rough isometry and let $\psi$ be a rough inverse for $\phi$. If $f \in D_p(G)$, then $\lim_{|g| \rightarrow \infty} | f((\psi \circ \phi )(g)) - f(g) | = 0$.
\end{Lem}
\begin{proof}
Let $g \in G$, since $\psi$ is a rough inverse of $\phi$ there are non-negative constants $a, b$ and $c$ with $a \geq 1$ such that $d_G( (\psi \circ \phi) (g), g) \leq a(c + b)$. Let $g_1, g_2, \dots , g_n$ be a path in $X_G$ of length not more than $a (c+b)$ with $g_1 = g$ and $g_n = (\psi \circ \phi) (g)$. So
\begin{equation*}
\begin{split}
| f((\psi \circ \phi) (g)) - f(g) |^p & = |\sum_{k=1}^{n-1} ( f(g_{k+1}) - f(g_k))|^p \\
                              & \leq n^{p-1} \sum_{k = 1}^{n-1} | f(g_{k+1}) - f(g_k) |^p.
\end{split}
\end{equation*} 
The last sum approaches zero as $|g| \rightarrow \infty$ since $f \in D_p(G)$ and $n \leq a(c+b)$. Thus $\lim_{|g| \rightarrow \infty} | f( (\psi \circ \phi) (g)) - f(g) | = 0$.
\end{proof}

The next proposition shows that $\bar{\phi}$ is a bijection.
\begin{Prop}
The function $\bar{\phi}$ is a bijection.
\end{Prop}
\begin{proof}
Let $x_1, x_2 \in \partial_p(G)$ such that $x_1 \neq x_2$ and let $f \in BD_p(G)$ with $f(x_1) \neq f(x_2)$. There exists sequences $(g_n)$ and $(g_n')$ in $G$ such that $(g_n) \rightarrow x_1$ and $(g_n') \rightarrow x_2$. Now assume that $\bar{\phi}(x_1) = \lim_{n \rightarrow\infty} (\phi(g_n)) = \lim_{n \rightarrow \infty} ( \phi(g_n')) = \bar{\phi}(x_2)$, so $\lim_{n \rightarrow \infty} f( (\psi \circ \phi)(g_n)) = \lim_{n \rightarrow \infty} f ((\psi \circ \phi) (g_n'))$. It follows from Lemma \ref{ontolemma} that $\lim_{n \rightarrow \infty} f(g_n) = \lim_{n \rightarrow \infty} f(g_n')$, thus $f(x_1) = f(x_2)$, a contradiction. Hence $\bar{\phi}$ is one-to-one.

We now proceed to show that $\bar{\phi}$ is onto. Let $ y \in \partial_p(H)$ and let $(h_n)$ be a sequence in $H$ that converges to $y$. By passing to a subsequence if necessary, we can assume that there exist an unique $x$ in the compact Hausdorff space $Sp( BD_p(G))$ such that $( \psi(h_n)) \rightarrow x$. Since $\lim_{n \rightarrow \infty} |h_n| \rightarrow \infty$ it follows that $\lim_{n \rightarrow \infty} |\psi (h_n)| \rightarrow \infty$, so $x \notin G$. Using an argument similar to the first paragraph in the proof of Proposition \ref{welldefined} we obtain $x \in \partial_p(G)$. The proof will be complete once we show that $\bar{\phi} (x) = y$. Let $f \in BD_p(H)$. By Lemma \ref{ontolemma} we see that $\lim_{n \rightarrow \infty} |f( (\phi \circ \psi) (h_n)) - f(h_n) | = 0$. Thus $f(\bar{\phi} (x) ) = f(y)$ for all $f \in BD_p(H)$. Hence $\bar{\phi} (x) = y$.
\end{proof}

To finish the proof that the bijection $\bar{\phi}$ is a homeomorphism we only need to show that $\bar{\phi}$ is continuous, since both $Sp(BD_p(G))$ and $Sp(BD_p (H))$ are compact Hausdorff spaces. To see that $\bar{\phi}$ is continuous, let $V$ be a basic open neighborhood of $\partial_p (H)$. Then there exists a $D_p$-massive subset $A$ of $H$ with inner potential $v$ such that $V = \hat{A}_v \cap \partial_p (H)$. We may and do assume that $v = 0$ on $H\backslash A$. Observe that $y \in V$ if and only if $\hat{v}(y) >0$. By Lemma \ref{adjoint} $\phi^{\ast} v = v \circ \phi \in BD_p(G)$. Combining Proposition \ref{decomp} and Theorem \ref{vanishbound} we see that there exists an $h \in BHD_p(G)$ such that $h = \hat{v} \circ \bar{\phi}$ on $\partial_p(G)$. Let $x \in \bar{\phi}^{-1} (V)$. Then $h(x) = (v \circ \phi)(x) = \hat{v} (\bar{\phi}(x)) > 0$. Let $U = \{ x' \in \partial_p(G) \mid h(x') > \frac{h(x)}{2} \}$. The set $U$ is open since $h$ is continuous on $\partial_p(G)$. Now $\hat{v}(\bar{\phi} (x')) > 0$ for each $x' \in U$, so $\bar{\phi}(x') \in V$, thus $U \subseteq \bar{\phi}^{-1} (V)$. Therefore, $\bar{\phi}^{-1} (V)$ is open and the proof that $\bar{\phi}$ is a homeomorphism is complete.

We will now prove Theorem \ref{bijection}. Let $\phi$ be a rough isometry from $G$ to $H$ and let $\psi$ be a rough inverse to $\phi$. Let $h \in BHD_p(G)$. By Lemma \ref{adjoint} $h \circ \psi \in BD_p(H)$. Let $\pi(h \circ \psi)$ be the unique element in $BHD_p(H)$ given by Proposition \ref{decomp}. We now define a map $\Phi \colon BHD_p (G) \mapsto BHD_p(H)$ by $\Phi (h) = \pi(h \circ \psi).$ Theorem \ref{vanishbound} implies that $\pi (h \circ \psi)(\bar{\phi} (x)) = (h \circ \psi)(\bar{\phi}(x))$ for all $x \in \partial_p(G)$, where $\bar{\phi}$ is the homeomorphism from $\partial_p(G)$ to $\partial_p(H)$ defined earlier in this section. Thus $\Phi (h) (\bar{\phi}(x)) = (h \circ \psi)(\bar{\phi}(x)) = h(x)$ for all $x \in \partial_p(G)$. We can now show that $\Phi$ is one-to-one. Let $h_1, h_2 \in BHD_p(G)$ and suppose that $\Phi (h_1) = \Phi (h_2)$. So $\Phi (h_1) (\bar{\phi}(x)) = \Phi (h_2) ( \bar{\phi}(x))$ for all $ x \in \partial_p (G)$, which implies $h_1(x) = h_2(x)$ for all $x \in \partial_p (G)$. Hence, $h_1 = h_2$ by Proposition \ref{boundval}. Thus $\Phi$ is one-to-one.

We will now show that $\Phi$ is onto. Let $v \in BHD_p(H)$. Then $v \circ \phi \in BD_p(G)$. Let $h = \pi(v \circ \phi)$, where $\pi (v \circ \phi)$ is the unique element in $BHD_p(G)$ given by Proposition \ref{decomp}. Let $y \in \partial_p (H)$. Since $h(x) = \pi( v \circ \phi)(x)$ for all $ x \in \partial_p(G)$ and $\bar{\psi} \circ \bar{\phi}$ equals the identity on $\partial_p(G)$, we see that $(\Phi (h))(y) = \pi (h \circ \psi)(y) = h (\psi(y)) = v(\phi \circ \psi)(y) = v(y)$. Thus $\Phi$ is onto and the proof of Theorem \ref{bijection} is complete.

The map $\Phi$ is an isomorphism in the case $p=2$ since $BHD_2 (G)$ and $BHD_2 (H)$ are linear spaces. However, in general these spaces are not linear if $ p \neq 2$. 

\section{Translation Invariant Linear Functionals}\label{tilf}
By definition we have the following inclusions:
$$ \mbox {Diff}(\ell^p(G)) \subseteq \mbox{Diff}(D_p(G)/\mathbb{C}) \subseteq \ell^p(G) \subseteq D_p(G)/\mathbb{C}.$$
The set $D_p(G)/\mathbb{C}$ is a Banach space under the norm induced from the pseudonorm on $D_p(G)$. Thus if $[f]$ if a class from $D_p(G) / \mathbb{C}$, then its norm is given by 
$$ \parallel [f] \parallel_{D(p)} = \left( \sum_{g \in G} \sum_{s \in S} \vert f(gs^{-1}) - f(g) \vert^p \right)^{1/p}.$$
We will write $\parallel f \parallel_{D(p)}$ for $\parallel [f] \parallel_{D(p)}$. Now $(\overline{\ell^p (G)})_{D(p)} = D_p(G)/\mathbb{C}$ if and only if $( \overline{ \ell^p(G) \oplus \mathbb{C}})_{D(p)} = D_p(G).$ So $\bar{H}_{(p)}^1 (G) = 0$ if and only if $( \overline{\ell^p (G)})_{D(p)} = D_p(G)/\mathbb{C}$. We begin by proving the following
\begin{Lem} \label{equality}
$(\overline{\mbox{Diff}(D_p(G)/\mathbb{C})})_{D(p)} = (\overline{\ell^p(G)})_{D(p)}$.
\end{Lem}
\begin{proof}
Let $f \in \ell^p(G)$. By \cite[Lemma 1]{Woodward} there is a sequence $( f_n )$ in $\mbox{Diff}(\ell^p(G))$ that converges to $f$ in the $\ell^p$-norm. It follows from Minkowski's inequality that for $s \in S,\parallel (f-f_n)_s - (f-f_n) \parallel_p^p = \sum_{g \in G} | f(gs^{-1}) - f_n (gs^{-1}) - (f(g) - f_n(g))|^p \rightarrow 0$ as $n \rightarrow \infty$. Hence $f \in (\overline{\mbox{Diff}(\ell^p(G))})_{D(p)}$ which implies $\ell^p(G) \subseteq ( \overline{\mbox{Diff}(\ell^p(G))})_{D(p)}$. The result now follows.
\end{proof}

We are now ready to prove the following characterization of nonzero continuous TILFs on $D_p(G)/\mathbb{C}$.
\begin{Thm} \label{charlptilf}
Let $1 < p \in \mathbb{R}$. Then $\bar{H}^1_{(p)}(G) \neq 0$ if and only if there exists a nonzero continuous TILF on $D_p(G)/\mathbb{C}$.
\end{Thm}
\begin{proof}
If $\bar{H}^1_{(p)}(G) \neq 0$, then $(\overline{\ell^p(G)})_{D(p)} \neq D_p(G)/\mathbb{C}$. It now follows from the Hahn-Banach theorem that there exists a nonzero continuous linear functional $T$ on $D_p(G)/\mathbb{C}$ such that $(\overline{\ell^p(G)})_{D(p)}$ is contained in the kernel of $T$. Thus $T$ is translation invariant by Lemma \ref{equality}.

Conversely if $T$ is a continuous TILF on $D_p(G)/\mathbb{C}$, then $T(f) = 0$ for all $f \in (\overline{\ell^p(G)})_{D(p)}$. So if there exists a nonzero continuous TILF on $D_p(G)/\mathbb{C}$, then $(\overline{\ell^p(G)})_{D(p)} \neq D_p(G)/\mathbb{C}$.
\end{proof}

Theorem \ref{chartilf} now follows by combining Theorems \ref{charlptilf} and \ref{charvanishing}.

If $h \in D_p(G)/\mathbb{C}$, then $\langle \triangle_p h, \cdot \rangle$ is a well-defined continuous linear functional on $D_p(G)/\mathbb{C}$ since equivalent functions in $D_p(G)/\mathbb{C}$ differ by a constant. It was shown in Proposition 3.4 of \cite{Puls} that if $h \in HD_p(G)/ \mathbb{C}$ and $f \in ( \overline{\ell^p(G)})_{D(p)}$, then $\langle \triangle_p h, f \rangle = 0$. Consequently, if $h \in HD_p(G)/\mathbb{C}$, then $\langle \triangle_p h, \cdot \rangle$ defines a continuous TILF on $D_p(G)/ \mathbb{C}$. Thus there are no nonzero continuous TILFs on $D_p(G)/ \mathbb{C}$ when $HD_p(G)$ only contains the constant functions.

If $\bar{H}_{(p)}^1(G) = 0$, then $(\overline{\ell^p(G)})_{D(p)} = D_p(G)/\mathbb{C}$. It is known that $\ell^p(G)$ is closed in $D_p(G)/\mathbb{C}$ if and only if $G$ is nonamenable, \cite[Corollary 1]{guichardet}. As was mentioned in Section \ref{Mainresults}, if $G$ is nonamenable, then zero is the only TILF on $\ell^p(G)$. Consequently zero is the only TILF on $D_p(G)/\mathbb{C}$ when $G$ is nonamenable and $\bar{H}_{(p)}^1 (G) =0$. Summing up we obtain:
\begin{Thm}
Let $G$ be an infinite, finitely generated group and let $1 < p \in \mathbb{R}$. The following are equivalent
\begin{enumerate}
\item $\bar{H}^1_{(p)} (G) = 0$
\item $\mbox{Either } \partial_p(G) = \emptyset \mbox{ or } \#(\partial_p(G)) = 1$
\item $HD_p(G) = \mathbb{C}$
\item $BHD_p(G) = \mathbb{C}$
\item The only continuous TILF on $D_p(G)/\mathbb{C}$ is zero\\
If moreover $G$ is nonamenable, then this is still equivalent to:
\item Zero is the only TILF on $D_p(G)/\mathbb{C}$
\end{enumerate}
\end{Thm}

We will now give some examples that show zero is not the only $TILF$ on $D_p(G)/\mathbb{C}$ when $G$ is nonamenable, this differs from the $\ell^p(G)$ case. It was shown in \cite[Corollary 4.3]{Puls} that $\bar{H}^1_{(p)} (G) \neq 0$ for groups with infinitely many ends and $1 < p \in \mathbb{R}$. Thus by Theorem \ref{charlptilf} there exists a nonzero continuous TILF on $D_p(G)/\mathbb{C}$. Another question that now arises is: if there is a nonzero continuous TILF on $D_r(G)/\mathbb{C}$ for some nonamenable group $G$ and some real number $r$, then is it true that there is a nonzero continuous TILF on $D_p(G)/\mathbb{C}$ for all real numbers $p>1$? The answer to this question is no. To see this let $\mathcal{H}^n$ denote hyperbolic $n$-space, and suppose $G$ is a group that acts properly discontinuously on $\mathcal{H}^n$ by isometries and that the action is cocompact and free. By combining Theorem 2 of \cite{BMV} and Theorem 1.1 of \cite{PulsADM} we obtain $\bar{H}^1_{(p)}(G) \neq 0$ if and only if $p > n-1$.

\bibliographystyle{plain}
\bibliography{pboundlp2}
\end{document}